\documentclass{amsart}

\usepackage{amssymb,amsfonts, amsmath, amsthm}
\usepackage{mathrsfs}
\usepackage{mathtools}
\usepackage[mathscr]{euscript}
\usepackage{tikz}
\usepackage{tikz-cd}
\usepackage{mathpazo}
\usepackage{bbold}
\usepackage[bbgreekl]{mathbbol}

\usepackage[pagebackref, colorlinks=true, linkcolor=blue, citecolor = red]{hyperref}
\renewcommand*{\backref}[1]{}
\renewcommand*{\backrefalt}[4]{({%
		\ifcase #1 Not cited.%
		\or On p.~#2%
		\else On pp.~#2%
		\fi%
	})}

\usepackage[nameinlink,capitalise,noabbrev]{cleveref}
\crefname{subsection}{Subsection}{Subsection}

\usetikzlibrary{patterns}

\usepackage{geometry}
\DeclareMathAlphabet{\mathbbe}{U}{bbold}{m}{n}

\def\DDelta{{\mathbbe{\Delta}}}
\newcommand{\DD}{\DDelta}

\newcommand{\C}{\mathscr{C}}
\newcommand{\cC}{\mathcal{C}}
\newcommand{\kC}{\mathfrak{C}}
\newcommand{\D}{\mathscr{D}}
\newcommand{\cK}{\mathcal{K}}
\newcommand{\s}{\mathscr{S}}
\newcommand{\cM}{\mathcal{M}}
\newcommand{\K}{\mathscr{K}}

\newcommand{\set}{\mathscr{S}\mathrm{et}}
\newcommand{\sset}{s\mathscr{S}\mathrm{et}}
\newcommand{\sS}{s\s}

\newcommand{\Fun}{\mathrm{Fun}}
\newcommand{\St}{\mathrm{St}}
\newcommand{\Un}{\mathrm{Un}}
\newcommand{\sSt}{s\St}
\newcommand{\sUn}{s\Un}
\newcommand{\scUn}{s\mathcal{U}\mathrm{n}}

\newcommand{\RFib}{\mathcal{R}\mathcal{F}\mathrm{ib}}
\newcommand{\Cart}{\mathcal{C}\mathrm{art}}
\newcommand{\cFun}{\mathcal{F}\mathrm{un}}
\newcommand{\Kan}{\mathcal{K}\mathrm{an}}
\newcommand{\Catinfty}{\mathcal{C}\mathrm{at}_{\infty}}
\newcommand{\Comp}{\mathcal{C}\mathrm{omp}}
\newcommand{\CSS}{\mathcal{C}\mathcal{S}\mathcal{S}}
\newcommand{\QCat}{\mathcal{Q}\mathcal{C}\mathrm{at}}
\newcommand{\Seg}{\mathcal{S}\mathrm{eg}}
\newcommand{\Nerve}{\mathcal{N}\mathrm{erve}}
\newcommand{\Und}{\mathcal{U}\mathrm{nd}}

\newcommand{\Hom}{\mathrm{Hom}}
\newcommand{\id}{\mathrm{id}}

\newtheorem{theorem}[equation]{Theorem}
\newtheorem{corollary}[equation]{Corollary}

\theoremstyle{remark}
\newtheorem{conjecture}[equation]{Conjecture}

\numberwithin{equation}{section}

\title{Cosmological Unstraightening}

\date{May 2025}

\author{Nima Rasekh}

\address{Institut f{\"u}r Mathematik und Informatik, Universit{\"a}t Greifswald, Greifswald, Germany}

\email{nima.rasekh@uni-greifswald.de}



\begin{document}

\begin{abstract}
 The unstraightening construction due to Lurie \cite{lurie2009htt} establishes an equivalence between presheaves and fibrations, using one prominent model of $(\infty,1)$-categories, namely quasi-categories. In this work we generalize this result by proving that for all $\infty$-cosmoi of $(\infty,1)$-categories in the sense of Riehl and Verity \cite{riehlverity2022elements}, which includes quasi-categories but also complete Segal spaces or $1$-complicial sets, their corresponding notions of fibrations and presheaves are biequivalent $\infty$-cosmoi via a natural zig-zag of cosmological biequivalences. The major idea that makes this possible is a lift of the quasi-categorical unstraightening construction to a cosmological biequivalence.
\end{abstract}

\maketitle

\section{Introduction}
The theory of $(\infty,1)$-categories has played a key conceptual role in many aspects of mathematics, including topology, geometry and representation theory. This has necessitated a rigorous development, which has primarily followed two strands. The \emph{model-dependent} approach uses specific models, such as quasi-categories or complete Segal spaces, to develop the desired theory, whereas the \emph{model-independent approach}, such as $\infty$-cosmoi \cite{riehlverity2022elements} or the axiomatization of \cite{barwickschommerpries2021unicity}, focus on universal properties. Often one starts with a model-dependent definition, benefiting from the point-set properties of the model, only to then later develop a more universal approach, which generalizes the results to all models. An elegant example is given by \emph{adjunctions of $(\infty,1)$-categories}, which were first defined for quasi-categories strongly relying on properties this model provides \cite[Definition 5.2.2.1]{lurie2009htt}, but then redeveloped in the context of $\infty$-cosmoi, using $2$-categorical ideas \cite[Definition 2.1.1]{riehlverity2022elements} in a way that generalizes the original approach \cite[F.5]{riehlverity2022elements}. 

For some quasi-categorical notions their model-independent analogues, while studied, have remained less understood. One interesting example is the unstraightening construction. It has first been introduced by Lurie \cite{lurie2009htt} and subsequently studied by many others \cite{heutsmoerdijk2015leftfibrationi,stevenson2017covariant,riehlverity2018comprehension,cisinski2019highercategories,cisinskinguyen2022straightening}, primarily using quasi-categories. Generally speaking it provides an equivalence between (right or Cartesian) fibrations and presheaves (valued in spaces or $\infty$-categories). It is a key ingredient in a variety of very important higher categorical constructions, such as monoidal structures. While both fibrations and presheaves provide us with examples of $\infty$-cosmoi, the unstraightening construction itself does not satisfy the required technical conditions (as we observe in \cref{subsec:un rfib}). This means the transition of the unstraightening construction from a model dependent approach via quasi-categories to a model independent approach via $\infty$-cosmoi has remained incomplete.

In this work we lift the unstraightening functor to a \emph{cosmological biequivalence} between $\infty$-cosmoi. The key technical insight that makes this possible is the fibrational version of the \emph{universal simplicial model category}, which is a concept introduced by Rezk, Schwede and Shipley \cite{rss2001simplicialmodel}. It was already observed by Joyal and Tierney that the universal simplicial model structure on quasi-categories is given by complete Segal spaces \cite[Section 8]{joyaltierney2007qcatvssegal}, and this was generalized by the author to a relation between Cartesian fibrations on marked simplicial sets and Cartesian fibrations on marked simplicial spaces \cite{rasekh2021cartfibmarkedvscso}. We use this insight to lift the unstraightening construction from (marked) simplicial sets to (marked) simplicial spaces, making it simplicial along the way, fulfilling the precise technical requirement to obtain cosmological biequivalences
\[\scUn_{\QCat,S}\colon \cFun(\kC[S]^{op},\Kan) \to \RFib_{\QCat}(S), \hspace{0.3in} \scUn^+_{\QCat,S}\colon \cFun(\kC[S]^{op},1{-}\Comp)^{core} \to \Cart_{\QCat}(S)^{core}\]
from the $\infty$-cosmoi of (space or $\infty$-category valued) presheaves, to the $\infty$-cosmoi of (right or Cartesian) fibrations (\cref{cor:main}/\cref{cor:main marked}). Building on this result, we show that for every $\infty$-cosmos of $(\infty,1)$-categories $\cK$ (which includes complete Segal spaces, Segal categories and $1$-complicial sets) their notion of fibrations and functors are cosmologically biequivalent 
\[\cFun(\kC_{\cK}[S]^{op},\Kan) \simeq \RFib_{\cK}(S), \hspace{0.3in}  \cFun(\kC_{\cK}[\C]^{op},1{-}\Comp)^{core} \simeq \Cart_{\cK}(\C)^{core}\]
via natural zig-zags of cosmological biequivalences (\cref{cor:cosmological}/\cref{cor:cosmological marked}). Finally, for the particular relevant choices of $\infty$-cosmoi, such as complete Segal spaces, one gets a direct unstraightening cosmological biequivalence (\cref{cor:cosmological special}/\cref{cor:cosmological special marked})
\[\scUn_{\cK,S}\colon \cFun(\kC[S]^{op},\Kan) \to \RFib_{\QCat}(S), \hspace{0.3in} \scUn^+_{\cK,S}\colon \cFun(\kC[S]^{op},1{-}\Comp)^{core} \to \Cart_{\QCat}(S)^{core}\]
Note that while \ref{cor:main},\ref{cor:cosmological} and \ref{cor:cosmological special} are the best possible result, \ref{cor:main marked}, \ref{cor:cosmological marked} and \ref{cor:cosmological special marked} could theoretically permit a generalization via an enrichment over quasi-categories. Obtaining such a result currently faces significant theoretical challenges, further discussed in \cref{sec:warning}.

The cosmological biequivalences we obtain also have further computational benefits. Indeed, we often want to analyze the straightening-unstraightening adjunction at the level of underlying $\infty$-categories. We can always obtain an adjoint equivalence of quasi-categories out of every Quillen equivalence, however the general construction is computationally challenging \cite{lowmazelgee2015fraction}. On the other hand, a \emph{simplicially enriched} Quillen adjunction of simplicially enriched model categories permits a very direct description of the underlying adjoint equivalence of quasi-categories via the homotopy coherent nerve \cite[Proposition 5.2.4.6]{lurie2009htt}, which we present in \cref{cor: cor} and \cref{cor:cor marked}.

\subsection{Acknowledgment}
I want to thank Raffael Stenzel for helpful discussions which motivated this project in the first place. I also thank Lyne Moser, Martina Rovelli, Viktoriya Ozornova and Gijs Heuts for helpful discussions regarding model structures on double categories and unstraightening. Moreover, I thank the referee for many helpful comments and particularly for catching an error in an earlier draft. Finally, I also thank the Max Planck Institute for Mathematics in Bonn for its hospitality and financial support.

\section{Reviewing Concepts}
In this section we present a high level overview of the pertinent results from the theory of (marked) simplicial sets and spaces, their model structures and straightening construction as well as relevant concepts regarding $\infty$-cosmoi, leaving precise definitions and more thorough treatment to the references. Only we have chosen slightly different notation regarding $\infty$-cosmoi, as explained in \cref{subsec:cosmos}.

\subsection{Unstraightening right fibrations} \label{subsec:un rfib}
 Let $\sset$ denote the category of simplicial sets and note it comes with the Kan model structure, that we denote $\sset^{Kan}$. Fix a simplicial set $S$. Following \cite[Definition 1.1.5.1]{lurie2009htt}, the simplicially enriched category $\kC[S]$ denotes the categorification of $S$. We denote by $\Fun(\kC[S]^{op},\sset)$ the category of simplicial enriched functors and note it comes with a projective model structure we denote by $\Fun(\kC[S]^{op},\sset^{Kan})^{proj}$, which is in fact simplicial \cite[Proposition A.2.8.2]{lurie2009htt}.

The over-category $\sset_{/S}$ comes with the contravariant model structure \cite[Proposition 2.1.4.7, Remark 2.1.4.12]{lurie2009htt}, denoted $(\sset_{/S})^{contra}$ with fibrant objects \emph{right fibrations} $R \to S$ \cite[Proposition 2.1.4.9]{lurie2009htt}, and that is also simplicial \cite[Proposition 2.1.4.8]{lurie2009htt}. The \emph{straightening construction} \cite[Theorem 2.2.1.2]{lurie2009htt} gives us an equivalence between these two model structures
\begin{equation} \label{eq:straightening}
	\begin{tikzcd}
		(\sset_{/S})^{contra} \arrow[r, "\St_S", "\bot"', shift left=2] & \Fun(\kC[S]^{op},\sset^{Kan})^{proj} \arrow[l, "\Un_S", shift left=2]
	\end{tikzcd}.	
\end{equation}
 	While both model structures are simplicially enriched, the adjunction is not simplicial. Indeed, if $S=\Delta[0]$ then the contravariant model structure is the Kan model structure on $\sset$ \cite[Lemma 2.1.3.3]{lurie2009htt} and we have $\St_{\Delta[0]}(\Delta[0]) = \Delta[0]$. As every simplicially enriched left adjoint out of $\sset$ is uniquely determined by the value at $\Delta[0]$, $\St_{\Delta[0]}$ can only be simplicially enriched if it is the identity, which is not the case \cite[Remark 2.2.2.6]{lurie2009htt}, meaning the functor $\St_{\Delta[0]}$ is not simplicially enriched. 

\subsection{Unstraightening Cartesian fibrations}
 Let $\sset^+$ denote the category of \emph{marked simplicial sets} with objects pairs $(X,A)$ where $X$ is a simplicial set and $A \subseteq X_1$ such that $A$ includes degenerate $1$-cells. See \cite[Definition 3.1.0.1]{lurie2009htt} for more details. Fix a simplicial set $S$ and let $S^\#$ denote the marked simplicial set $(S,S_1)$. We denote by $(\sset^+_{/S^\#})^{Cart}$, which we simplify to $(\sset^+_{/S})^{Cart}$, the Cartesian model structure over $S$ \cite[Proposition 3.1.3.7]{lurie2009htt} with fibrant objects Cartesian fibrations \cite[Proposition 3.1.4.1]{lurie2009htt}, and note it is simplicial, meaning enriched over the Kan model structure \cite[Corollary 3.1.4.4]{lurie2009htt}. In the particular case of $S = \Delta[0]$, the Cartesian model structure $(\sset^+)^{Cart}$ has fibrant objects $(\C,\C^{equiv})$, where $\C$ is a quasi-category and $\C^{equiv}$ is the set of equivalences, and in fact forgetting the marking $(\sset^+)^{Cart} \to \sset^{Joy}$ is a Quillen equivalence \cite[Theorem 3.1.5.1]{lurie2009htt}. 
	The straightening construction mentioned above generalizes to the \emph{marked straightening construction} \cite[Theorem 3.2.0.1]{lurie2009htt}
 \begin{equation} \label{eq:straightening marked}
 	\begin{tikzcd}
 		(\sset^+_{/S})^{Cart} \arrow[r, "\St^+_S", "\bot"', shift left=2] & \Fun(\kC[S]^{op},(\sset^+)^{Cart})^{proj} \arrow[l, "\Un^+_S", shift left=2]
 	\end{tikzcd}.	
 \end{equation}
 Notice similar to the unmarked case, the adjunction $(\St^+_S,\Un^+_S)$ is not simplicial. Before we end this subsection let us note that the functor $\Un^+_S$ is in fact simplicially enriched, as explained in the the beginning of \cite[Subsection 3.2.4]{lurie2009htt}, and it is the failure of  $\Un^+_S$ to preserve simplicial cotensors, which prevents $(\St^+_S,\Un^+_S)$ from being a simplicially enriched adjunction.

\subsection{Marked simplicial spaces} \label{subsec:marked spaces}
 Let $\sS$ denote the category of bisimplicial sets, which we call simplicial spaces and fix a simplicial space $X$. The over-category $\sS_{/X}$ comes with the contravariant model structure, denoted $(\sS_{/X})^{contra}$ with fibrant objects \emph{right fibrations} $R \to X$, and that is also simplicial \cite[Theorem 3.12]{rasekh2023left}. Moreover, for a given simplicial set $S$, there is a Quillen equivalence of contravariant model structures \cite[Theorem B.12]{rasekh2023left}
 \begin{equation} \label{eq:contra}
 	\begin{tikzcd}
 		(\sset_{/S})^{contra} \arrow[r, "p_1^*", "\bot"', shift left=2] & (\sS_{/S})^{contra} \arrow[l, "i_1^*", shift left=2]
 	\end{tikzcd}.	
 \end{equation}
 Notice, as $p_1^*(T)$ (for $T$ a simplicial set) is a constant simplicial space \cite[Theorem 4.11]{joyaltierney2007qcatvssegal}, we will again denote it by $T$, simplifying notation.  
 
 Next, let $\sS^+$ denote the category of \emph{marked simplicial spaces}, with objects pairs $(X,A)$ with $X$ a simplicial space and $A \subseteq X_1$ a sub-simplicial set that includes the degenerate $1$-simplices. See \cite[Subsection 2.1]{rasekh2021cartfibmarkedvscso} for a more detailed discussion. Fix a simplicial space $X$ and let $X^\#$ denote the marked simplicial space $(X,X_1)$. We denote by $(\sS^+_{/X})^{Cart}$ the Cartesian model structure over $X$ with fibrant objects Cartesian fibrations (in the sense of \cite[Definition $2.40$]{rasekh2021cartfibmarkedvscso}), and note it is also simplicial \cite[Theorem 2.44]{rasekh2021cartfibmarkedvscso}. Finally, for a given simplicial set $S$, there is a Quillen equivalence of Cartesian model structures \cite[Theorem 2.47]{rasekh2021cartfibmarkedvscso}
  \begin{equation} \label{eq:contra marked}
 	\begin{tikzcd}
 		(\sset^+_{/S})^{Cart} \arrow[r, "(p^+_1)^*", "\bot"', shift left=2] & (\sS^+_{/S})^{Cart} \arrow[l, "(i^+_1)^*", shift left=2]
 	\end{tikzcd}.	
 \end{equation}

\subsection{\texorpdfstring{$\infty$}{oo}-Cosmoi} \label{subsec:cosmos}
$\infty$-cosmoi are categories strictly enriched over quasi-categories with finite weighted limits and a notion of fibration, called \emph{isofibration} \cite[Definition 1.2.1]{riehlverity2022elements}. The guiding example is the $\infty$-cosmos of quasi-categories ($\QCat$) \cite[Proposition 1.2.10]{riehlverity2022elements}, but it also includes other models such as complete Segal spaces ($\CSS$), Segal categories ($\Seg$) and $1$-complicial sets ($1{-}\Comp$) \cite[Example 1.2.24]{riehlverity2022elements}. These $\infty$-cosmoi can all be obtained by starting with a model structure enriched over the Joyal model structure in which all fibrant objects are cofibrant, and then restricting to the full enriched subcategory of fibrant objects \cite[Proposition. E.1.1]{riehlverity2022elements}, and hence are called the \emph{underlying $\infty$-cosmos} of the	model structure.
	
The $\infty$-cosmoi from the previous paragraph are in fact all equivalent to the $\infty$-cosmos of quasi-categories \cite[Example 1.3.9]{riehlverity2022elements}, in the following precise sense. First recall a cosmological functor of $\infty$-cosmoi is an enriched functor that preserves fibrations and limits \cite[Definition 1.3.1]{riehlverity2022elements} and it is a cosmological biequivalence if it is fully faithful and essentially surjective \cite[Definition 1.3.8]{riehlverity2022elements}. Now, for each one of these examples $\cK$, the functor $\Hom_{\cK}(1,-)\colon \cK \to \QCat$ that evaluates at the terminal object and we denote by $\Und$ (as it gives us the \emph{underlying quasi-category}) is a cosmological biequivalence. Given their importance such $\infty$-cosmoi are called \emph{$\infty$-cosmoi of $(\infty,1)$-categories} \cite[Definition 1.3.10]{riehlverity2022elements}. Notice, prominent examples of $\infty$-cosmoi of $(\infty,1)$-categories also have explicitly known inverse biequivalences $\Nerve\colon \QCat \to \CSS $, $\Nerve\colon \QCat \to \Seg$ and $(-)^{\natural}\colon \QCat \to 1{-}\Comp$.
 
As every Kan complex is in particular a quasi-category, appropriately chosen Kan enriched categories also give us $\infty$-cosmoi, called \emph{discrete} \cite[Definition 1.2.26]{riehlverity2022elements}. The key example of a discrete $\infty$-cosmos is the category of Kan complexes ($\Kan$) \cite[Proposition 1.2.12]{riehlverity2022elements}. For a given $\infty$-cosmos $\cK$, we can restrict each mapping quasi-category to its maximal underlying Kan complex (also called the core \cite[Definition 12.1.8]{riehlverity2022elements}), to obtain the discrete $\infty$-cosmos $\cK^{core}$.
 
For every $\infty$-cosmos $\cK$, there is a model independent way to study various fibrations over an $\infty$-category $\C$ in $\cK$. As a result we can define the $\infty$-cosmos of Cartesian fibrations over $\C$, which we denote by $\Cart_{\cK}(\C)$ \cite[Definition 5.2.1]{riehlverity2022elements}, and its sub-$\infty$-cosmos of right fibrations $\RFib_{\cK}(\C)$ (there called \emph{discrete Cartesian fibrations} \cite[Definition 5.5.3]{riehlverity2022elements} as explained in  \cite[Digression 5.5.4]{riehlverity2022elements}). 

Let $\cC$ be a simplicially enriched category. We want to construct the $\infty$-cosmos of functors from $\cC$ with values in spaces (or $\infty$-categories). More explicitly, we would like to construct the $\infty$-cosmoi $\Fun(\cC, \Kan)$ and $\Fun(\cC,1{-}\Comp)$. Unfortunately, due to the strict enrichment, the naive functor categories would not be $\infty$-cosmoi and would hence have to be suitably derived. Fortunately, the $\infty$-cosmoi $\Kan$ and $1{-}\Comp$ both arise as the underlying $\infty$-cosmos of a model structure, and in those cases the functor $\infty$-cosmos admits a more direct construction, which we discuss next. 

Let $M$ be a combinatorial model category enriched over the Joyal model structure with all objects cofibrant, and denote its underlying $\infty$-cosmos (the full subcategory of fibrant objects) by $\cM$. Then the category of enriched functors $\Fun(\cC,M)^{inj}$ with the injective model structure is again enriched over the Joyal model structure with all objects cofibrant \cite[Proposition A.3.3.2]{lurie2009htt}, and hence has an underlying $\infty$-cosmos, which we denote by $\cFun(\cC,\cM)$.

We can now apply this construction of the functor $\infty$-cosmos to our situation of interest. For a given $\infty$-cosmos of $(\infty,1)$-categories $\cK$ and $\infty$-category $\C$ in $\cK$, define the simplicially enriched category $\kC_\cK[\C]$ as $\kC_\cK[\C] = \kC[\Und\C]$. Following the discussion in the previous paragraph, we now define the functor $\infty$-cosmoi $\cFun(\kC_\cK[\C]^{op}, \Kan)$ and $\cFun(\kC_\cK[\C]^{op}, 1{-}\Comp)$ as the underlying $\infty$-cosmoi of the injective model structures on strict functor categories $\Fun(\kC_\cK[\C]^{op}, \sset^{Kan})^{inj}$ and $\Fun(\kC_\cK[\C]^{op}, (\sset^+)^{Cart})^{inj}$.
	
By \cite[Corollary 10.3.7]{riehlverity2022elements}, all these constructions are stable under cosmological biequivalences, meaning for a given cosmological biequivalence $F\colon\cK \to \cK'$ we get cosmological biequivalences $F\colon\Cart_{\cK}(\C) \to \Cart_{\cK'}(F\C)$ and $F\colon\RFib_{\cK}(\C) \to \RFib_{\cK'}(F\C)$. Additionally, combining this with \cite[Remark 2.1.4.11, Theorem 3.2.0.1]{lurie2009htt}, an equivalence $f\colon \C \to \D$ of $\infty$-categories in an $\infty$-cosmos $\K$ induces cosmological biequivalences $f^*\colon\Cart_{\cK}(\D) \to \Cart_{\cK}(\C)$ and $f^*\colon\RFib_{\cK}(\D) \to \RFib_{\cK}(\C)$ via pullback. Moreover, every cosmological biequivalence comes with a quasi-pseudoinverse \cite[Proposition 10.4.16]{riehlverity2022elements}, which also preserves and reflects all properties stated in \cite[Proposition 10.3.6]{riehlverity2022elements}. Finally, the right Quillen functor of every simplicially enriched Quillen adjunctions (equivalence) gives us such a cosmological functor (biequivalence) \cite[Corollary E.1.2]{riehlverity2022elements}, which means that $\Un_S$ and $\Un^+_S$ do not give us cosmological biequivalences. The best result regarding $\Un^+$ one can expect is a weaker statement involving natural equivalences, as established in \cite{stenzel2024unstraightening}.
 
\section{Cosmological Unstraightenings}
We now lift the straightening constructions \ref{eq:straightening} and \ref{eq:straightening marked} from simplicial sets to simplicial spaces, with the goal of getting cosmological biequivalences. For this section fix a simplicial set $S$. 

\subsection{A cosmological unstraightening for right fibrations}
We will commence with the unmarked case. Notice, there is an isomorphism of categories $\sS_{/S} \cong \Fun((\DD_{/S})^{op} \times \DD^{op},\set) \cong \Fun((\DD_{/S})^{op},\sset)$. Moreover, under this isomorphism of categories, the functor $p_1^*\colon \Fun((\DD_{/S})^{op},\set) \to \Fun((\DD_{/S})^{op},\sset)$, introduced in \ref{eq:contra}, is simply given by post-composition with the evident inclusion $\set \to \sset$. We can hence conclude that a simplicially enriched functor out of $\sS_{/S}$ is uniquely determined on the full subcategory $\DD_{/S}$ and so there is a unique simplicially enriched lift
\begin{equation} \label{eq:lift}
	\begin{tikzcd}
		\sset_{/S} \arrow[r, "p_1^*"] \arrow[dr, "\St_S"'] & \sS_{/S} \arrow[d, "\sSt_S"] \\
		& {\Fun (\kC[S]^{op},\sset)}
	\end{tikzcd},
\end{equation}
which for a given map of simplicial sets $T \to S$ and $n \geq 0$ satisfies $\sSt_S(T \times \Delta[n] \to S)= \St_S(T \to S) \times \{ \Delta[n]\}$, where $\{\Delta[n]\}$ denotes the constant functor with value $\Delta[n]$. 

By construction $\sSt_S$ preserves colimits and hence has a right adjoint $\sUn_S\colon \Fun(\kC[S]^{op},\sset) \to \sS_{/S}$. As, by definition, we have $\sSt_S(p_1)^* = \St_S$, by uniqueness of right adjoints we also have $(i_1)^*\sUn_S = \Un_S$. We now have the following main result.

\begin{theorem} \label{the:main}
	$(\sSt_S,\sUn_S)$ is a simplicially enriched Quillen equivalence between the contravariant model structure and the projective model structure.
\end{theorem}

\begin{proof}
	First we show the functor $\sUn_S$ takes fibrant objects to Reedy fibrations. We need to show that $\sSt_S$ takes morphisms of the form $(\partial F(n) \to F(n)) \square (A \to B)$ over $S$ to a trivial cofibration, where $A \to B$ is a trivial cofibration of simplicial sets. As the projective model structure is enriched, it suffices to show that $\sSt_S(\partial F(n) \to F(n))$ is a cofibration. However, by definition we have $\sSt_S(\partial F(n) \to F(n))=\St_S(\partial \Delta[n] \to \Delta[n])$ and so this holds by \ref{eq:straightening}.
	
	Now, for a given fibrant object $F\colon\kC[S]^{op} \to \sset$, $\sUn F$ is a Reedy fibration with $\Un_S F = i_1^* \sUn_S F$ a right fibration and so, by \cite[Theorem 1.2]{rasekh2023cartfibcss} and \cite[Remark B.21]{rasekh2023left}, $\sUn_SF$ is also a right fibrations. Finally, $p_1^*$ is a right Quillen equivalence and so reflects weak equivalences and fibrations between fibrant objects. Hence, the fact that $\Un_S$ preserves weak equivalences and fibrations between fibrant objects and \ref{eq:lift} imply that $\sUn_S$ does the same. By \cite[Proposition 7.15]{joyaltierney2007qcatvssegal}, this proves that $\sUn_S$ is right Quillen.
	
	Finally, by \ref{eq:straightening}, \ref{eq:contra} and $2$-out-of-$3$, $(\sSt_S,\sUn_S)$ is a simplicially enriched Quillen equivalence.
\end{proof}

We can now state the desired corollaries. First of all, the equivalence of the injective and projective model structure via the identity adjunction gives us the simplicially enriched Quillen equivalence
\[
	\begin{tikzcd}[column sep = 2cm]
		(\sS_{/S})^{contra}  \arrow[r, shift left = 2, "\sSt_S" near end] \arrow[rr, bend left = 17, "\sSt_S"] & 
		\Fun(\kC[S]^{op},\Kan)^{proj} \arrow[r, shift left = 2, "\id" near start] \arrow[l, shift left = 2, "\bot"', "\sUn_S" near start] & 
		\Fun(\kC[S]^{op},\Kan)^{inj} \arrow[l, shift left = 2, "\bot"', "\id" near end] \arrow[ll, bend left = 17, "\sUn_S"]
	\end{tikzcd},
\]
which induces a functor of Kan-enriched categories $\scUn_{\QCat,S}\colon \cFun(\kC[S]^{op},\Kan) \to \RFib_{\QCat}(S)$ immediately giving us the following result.
\begin{corollary} \label{cor:main}
	The unstraightening $\scUn_{\QCat,S}\colon \cFun(\kC[S]^{op},\Kan) \to \RFib_{\QCat}(S)$ is a cosmological biequivalence.
\end{corollary}

Now, we get the following by applying the homotopy coherent nerve to simplicially enriched adjunctions \cite[Proposition 5.2.4.6]{lurie2009htt}.

\begin{corollary} \label{cor: cor}
	The functor of quasi-categories $N\sUn_S\colon N(\cFun(\kC[S]^{op},\Kan)^\circ) \to N(\RFib(S)^\circ)$ is a right adjoint equivalence with left adjoint given by the left derived functor of $N\sSt_S$.% The adjunction $(N\sSt_S,N\sUn_S)$ is an adjoint equivalence of quasi-categories between the quasi-categories $N(\cFun(\kC[S]^{op},\Kan)^\circ)$ and $N(\RFib(S)^\circ)$.
\end{corollary}

Finally, we can generalize to more general $\infty$-cosmoi. Let $\cK$ be an $\infty$-cosmos of $(\infty,1)$-categories and $\C$ an $\infty$-category. \cref{cor:main} and our observations regarding cosmological biequivalences in \ref{subsec:cosmos} directly implies the following.

\begin{corollary} \label{cor:cosmological}
	The diagram of $\infty$-cosmoi
	\[\cFun(\kC_\cK[\C]^{op},\Kan) \xrightarrow{ \ \scUn_{\QCat, \Und \C} \ } \RFib_{\QCat}(\Und\C) \xleftarrow{ \ \Und \ } \RFib_{\cK}(\C).\] 
	is a natural zig-zag of cosmological biequivalences.
\end{corollary}

This in particular means that any choice of quasi-pseudoinverse of $\Und$ gives us a quasi-pseudofunctor biequivalence $\cFun(\kC_\cK[\C]^{op},\Kan) \to \RFib_{\cK}(\C)$ which preserves and reflects all relevant $\infty$-categorical properties. If we restrict to the case where $\cK$ is the $\infty$-cosmos of complete Segal spaces, Segal categories or $1$-complicial sets, we have explicit cosmological biequivalences in the other direction (\cref{subsec:cosmos}), which, to simplify notation, we denote by $\Und^{-1}\colon\QCat \to \cK$. In those three cases, for an $\infty$-category $\C$ in $\cK$, there is an equivalence in $\cK$, 
denoted $\eta_{\C}\colon\C \to \Und^{-1}\Und\C$, which gives us a cosmological biequivalence 
\[\eta_{\C}^*\colon \RFib_{\cK}(\Und^{-1}\Und\C)  \to \RFib_{\cK}(\C).\]
Combining these observations, we now get the following final result of this section.

\begin{corollary} \label{cor:cosmological special}
	Let $\cK= \CSS,\Seg, 1{-}\Comp$. Then $\scUn_{\cK,S}\colon \cFun(\kC_\cK[\C]^{op},\Kan) \to \RFib_{\cK}(\C)$ defined as 
	\[\cFun(\kC_\cK[\C]^{op},\Kan) \xrightarrow{ \ \scUn_{\QCat, \Und \C} \ } \RFib_{\QCat}(\Und\C) \xrightarrow{ \ \Und^{-1} \ } \RFib_{\cK}(\Und^{-1}\Und\C) \xrightarrow{ \eta_\C^*} \RFib_{\cK}(\C)\]
	is a cosmological biequivalence.
\end{corollary} 

\subsection{A cosmological unstraightening for Cartesian fibrations}
We now generalize this result from right fibrations to Cartesian fibrations. Let $\sSt^+_S$ be the unique simplicially enriched lift of $\St^+_S$ along $(p_1^+)^*$.
\begin{center}
	\begin{tikzcd}
		(\sset^+)_{/S} \arrow[r, "(p^+_1)^*"] \arrow[dr, "\St^+_S"'] & (\sS^+)_{/S} \arrow[d, "\sSt^+_S"] \\
		& {\Fun(\kC [S]^{op}, \sset^+)}
	\end{tikzcd}
\end{center}
Concretely, for a given simplicial set $T$ over $S$ we have $\sSt_S^+(T \times \Delta[n] \to S)= \sSt_S^+(T \to S) \times \{\Delta[n]^\#\}$, where $\{\Delta[n]^\#\}$ denotes the constant functor with value $\Delta[n]^\#$. Again, by construction $\sSt^+_S$ preserves colimits and hence has a right adjoint $\sUn_S^+\colon \Fun(\kC[S]^{op},\sset^+) \to (\sS^+)_{/S}$. Notice, as we have $\sSt^+_S(p_1^+)^* = \St^+_S$, by uniqueness of right adjoints we have $(i^+_1)^*\sUn_S^+ = \Un_S^+$. We can now prove the main result.

\begin{theorem} \label{the:main marked}
	$(\sSt^+_S,\sUn^+_S)$ is a simplicially enriched Quillen equivalence.
\end{theorem}

\begin{proof}
	We follow the same steps as in the proof of \cref{the:main}. Indeed, as $\St_S^+$ is left Quillen (by \ref{eq:straightening marked}) it takes $\partial \Delta[n] \to \Delta[n]$ over $S$ to a cofibration in $\Fun(\kC[S]^{op},\sset^+)^{proj}$ and so $\sUn^+_S$ takes fibrant objects to Reedy fibrations. Moreover, by \cite[Theorem 1.2]{rasekh2023cartfibcss}, \cite[Corollary 2.46]{rasekh2021cartfibmarkedvscso} and \cite[Theorem 3.1.5.1]{lurie2009htt} a Reedy fibration in $\sS^+_{/S}$ is a Cartesian fibration  if its image under $(i_1^+)^*$ is a Cartesian fibration and so $\sUn_S^+$ takes fibrant objects to Cartesian fibrations, as $\Un_S^+$ does. Similarly, $\sUn_S^+$ takes fibrations (weak equivalences) between fibrant objects to fibrations (weak equivalences) between Cartesian fibrations. Hence, by \cite[Proposition 7.15]{joyaltierney2007qcatvssegal}, $(\sSt^+_S,\sUn^+_S)$ is a Quillen adjunction. Finally, by \ref{eq:straightening marked}, \ref{eq:contra marked} and $2$-out-of-$3$ it is a simplicially enriched Quillen equivalence. 
\end{proof}

We can now state the desired corollaries. Similar to $\sUn_S$ above, $\sUn^+_S$ induces a functor of Kan-enriched categories $\scUn^+_{\QCat,S}\colon \cFun(\kC[S]^{op},1{-}\Comp)^{core} \to \Cart_{\QCat}(S)^{core}$ immediately giving us the following result.

\begin{corollary}\label{cor:main marked}
	The marked unstraightening $\scUn^+_{\QCat,S}\colon \cFun(\kC[S]^{op},1{-}\Comp)^{core} \to \Cart_{\QCat}(S)^{core}$ is a cosmological biequivalence. 
\end{corollary}

\begin{corollary}\label{cor:cor marked}
  	The functor of quasi-categories $N\sUn^+_S\colon N\cFun(\kC[S]^{op},\Catinfty)^{\circ} \to N\Cart(S)^{\circ}$ is a right adjoint equivalence with left adjoint given by the left derived functor of $N\sSt^+_S$.
\end{corollary}

We end this section with a generalization to $\infty$-cosmoi. Let $\cK$ be an $\infty$-cosmos of $(\infty,1)$-categories and $\C$ an $\infty$-category. \cref{cor:main marked} and our understanding of $\infty$-cosmoi now gives us the following.

\begin{corollary} \label{cor:cosmological marked}
	The diagram of $\infty$-cosmoi
	\[\cFun(\kC_\cK[\C]^{op},1{-}\Comp)^{core} \xrightarrow{ \ \scUn^+_{\Und \C} \ } \Cart_{\QCat}(\Und\C)^{core} \xleftarrow{ \ \Und \ } \Cart_{\cK}(\C)^{core}.\]
	is a natural zig-zag of cosmological biequivalences.
\end{corollary}

Again it follows that a choice of quasi-pseudoinverse of $\Und$ gives us a quasi-pseudofunctor biequivalence $\cFun(\kC_\cK[\C]^{op},\Kan)^{core} \to \Cart_{\cK}(\C)^{core}$ (preserving and reflecting $\infty$-categorical properties). Moreover, if $\cK$ is the $\infty$-cosmos of complete Segal spaces, Segal categories or $1$-complicial sets, the explicit inverse cosmological biequivalences (\cref{subsec:cosmos}), which we again simply denote by $\Und^{-1}$, and the analogous cosmological biequivalence induced via pullback
\[	\eta_{\C}^*\colon \Cart_{\cK}(\Und^{-1}\Und\C)^{core} \to \Cart_{\cK}(\C)^{core},\]
gives us the following final result of this section.

\begin{corollary} \label{cor:cosmological special marked}
 Let $\cK= \CSS,\Seg, 1{-}\Comp$. Then $\scUn^+_{\cK,S}\colon \cFun(\kC_\cK[\C]^{op},\Kan)^{core} \to \Cart_{\cK}(\C)^{core}$ defined as 
	\[\cFun(\kC_\cK[\C]^{op},1{-}\Comp)^{core} \xrightarrow{ \ \scUn^+_{\QCat, \Und \C} \ } \Cart_{\QCat}(\Und\C)^{core} \xrightarrow{ \ \Und^{-1} \ } \Cart_{\cK}(\Und^{-1}\Und\C)^{core} \xrightarrow{ \eta_\C^*} \Cart_{\cK}(\C)^{core}\]
	is a cosmological biequivalence.
\end{corollary} 

\section{Challenges Towards A Quasi-categorically enriched cosmological Unstraightening}\label{sec:warning}
The model structures on the categories $\Fun(\kC[S]^{op},\sset^+)$ and $\sset^+_{/S}$ can also be enriched over quasi-categories \cite[Remark 3.1.4.5]{lurie2009htt} and the straightening construction is non-strictly enriched as well \cite[Corollary 3.2.1.15]{lurie2009htt}. Hence, one could imagine a lift of the unstraightening construction 
\[\Un^+_S\colon \Fun(\kC[S]^{op},\sset^+) \to \sset^+_{/S},\] 
that is in fact strictly enriched (meaning cotensored) over quasi-categories.  However, existing results do not permit such generalization. Indeed, what we would need is a proof for the following conjecture.

\begin{conjecture}
	There exists a model structure on $\sS^+$ with the following specifications:
	\begin{enumerate}
		\item It is Quillen equivalent to the Joyal model structure.
		\item The inclusion $i^+_1\colon \sset^+ \to \sS^+$ is left Quillen from the Cartesian model structure.
		\item The inclusion $i_1\colon \sset \to \sS^+$ is left Quillen from the Joyal model structure.
	\end{enumerate}
\end{conjecture}
Assuming we had such a model structure, we could conceivably repeat the arguments in \cref{the:main marked} to obtain a quasi-categorically enriched lift for the unstraightening construction. 

On a first view, settling such a conjecture might appear straightforward. Indeed, the model structure on marked simplicial spaces described in \cref{subsec:marked spaces} almost satisfies these conditions, with the sole difference that $i_1\colon \sset \to \sS^+$ is left Quillen from the Kan model structure. However, as of yet such a model structure has remained out of reach. In fact, one way to obtain such a model structure would be to have a working model structure for double $\infty$-categories on bisimplicial sets which interacts well with the ``square construction", which takes an $(\infty,2)$-category to a double $\infty$-category. The existence was already conjectured by Gaitsgory and Rozenblyum \cite[Chapter 10, Theorems 4.1.3 and 5.2.3]{gaitsgoryrozenblyum2017dagI}. While the conjecture has now been settled in the case of strict double categories \cite[Theorem 6.1]{gmsv2023square}, the $\infty$-categorical case has remained open. 

\bibliographystyle{alpha}
\bibliography{main}

\end{document}